\documentclass[12pt]{article}
\usepackage{amsthm,wasysym,amssymb,eufrak,indentfirst,color,graphicx,cite,mathrsfs,bibspacing}
\usepackage[bookmarksnumbered, colorlinks, plainpages]{hyperref}
\begin{document}
\baselineskip 17pt
\newtheorem{theorem}{Theorem}[section]
\newtheorem{lemma}[theorem]{Lemma}
\newtheorem{proposition}[theorem]{Proposition}
\newtheorem{example}[theorem]{Example}
\newtheorem{definition}[theorem]{Definition}
\newtheorem{corollary}[theorem]{Corollary}
\newtheorem{remark}[theorem]{Remark}
\newtheorem{problem}[theorem]{Problem}
\title{\textbf{On the converse of Hall's theorem}\thanks{The author is supported by an NNSF of China (grant No. 11371335), the Start-up Scientific Research Foundation of Nanjing Normal University (grant No. 2015101XGQ0105) and a project funded by the Priority Academic Program Development of Jiangsu Higher Education Institutions.}}
\author{Xiaoyu Chen\\
{\small School of Mathematical Sciences and Institute of Mathematics, Nanjing Normal University,}\\ {\small Nanjing 210023, P. R. China}\\
 {\small E-mail: jelly@njnu.edu.cn}
}
\date{}
\maketitle
\begin{abstract}
In this paper, we mainly investigate the converse of a well-known theorem proved by P. Hall, and present detailed characterizations under the various assumptions of the existence of some families of Hall subgroups. In particular, we prove that if $p\neq 3$ and a finite group $G$ has a Hall $\{p,q\}$-subgroup for every prime $q\neq p$, then $G$ is $p$-soluble.
\end{abstract}
\renewcommand{\thefootnote}{\empty}

\section{Introduction}
\label{intro}

Throughout this paper, $p$ always denotes a prime, $\pi$ denotes a non-empty set of primes and $\mathbb{P}$ denotes the set of all primes. As usual, we denote the characteristic of the field $\mathbb{F}_q$ by $p$. For any positive integer $n$, $\pi(n)$ denotes the set of prime divisors of $n$ and $n_\pi$ denotes the $\pi$-part of $n$, that is, the largest $\pi$-number dividing $n$. Moreover, the symbols $G$, $|G|$ and $\pi(G)$ are used to denote a finite group, the order of $G$ and $\pi(|G|)$, respectively. \par

In accordance with the concepts introduced by P. Hall (see \cite{Hal}), a group $G$ is said to have the: (i) $E_\pi$-property if $G$ possesses a Hall $\pi$-subgroup; (ii) $C_\pi$-property if $G$ has the $E_\pi$-property and any two Hall $\pi$-subgroups of $G$ are conjugate in $G$; (iii) $D_\pi$-property if $G$ has the $C_\pi$-property  and every $\pi$-subgroup of $G$ is contained in some Hall $\pi$-subgroup of $G$. Let $E_\pi$ (resp. $C_\pi$, $D_\pi$) denote the class of all finite groups which have the $E_\pi$-property (resp. $C_\pi$-property, $D_\pi$-property). Clearly, $D_\pi\subseteq C_\pi\subseteq E_\pi$, and also there exist several examples to show that $E_\pi\neq C_\pi$ and $C_\pi\neq D_\pi$ (for example, see \cite[Examples 1.4 and 1.5]{Vdo}).\par

In 1872, M. L. Sylow proved an especially significant theorem (the so-called Sylow's theorem) to reveal that $G\in D_p$ for any finite group $G$ and any prime $p$. This famous result is one of the milestones in finite group theory, and stimulates a lot of research interest in establishing theorems of Sylow type. This research originates from P. Hall and S. A. Chunikhin's works \cite{Chu,Hal1,Hal2}, and attracts much attention  during
the past nearly a century. Among the recent fruitful works in this area, a highly surprising one is the completion of the classification of the Hall subgroups in finite known simple groups, which is mainly attributed to F. Gross, E. P. Vdovin and D. O. Revin. Also, in \cite{Rev1} (or see \cite{Rev2}), an  exhaustive description of finite groups which have the $D_\pi$-property was given by D. O. Revin. The readers can refer to a well-written survey \cite{Vdo} to acquire more information on the various criteria for the $E_\pi$, $C_\pi$ and $D_\pi$-property.

In this paper, we mainly focus on the converse of a well-known theorem which is proved by P. Hall in \cite{Hal}. Recall that a group $G$ is called $\pi$-separable if every composition factor (or equivalently, every chief factor) of $G$ is either a $\pi$-group or a $\pi'$-group. Also, $G$ is called $\pi$-soluble if every composition factor (or equivalently, every chief factor) of $G$ is either a $p$-group or a $\pi'$-group, where $p\in \pi$. From these definitions, $G$ is $p$-separable if and only if $G$ is $p$-soluble, and by the Feit-Thompson theorem, $G$ is $\pi$-separable if and only if $G$ is either $\pi$-soluble or $\pi'$-soluble. We now present P. Hall's Theorem as follows.

\begin{theorem}\label{h}$($see \textup{\cite[\textit{Corollary} E2.3]{Hal}} or \textup{\cite[\textit{Chap. }6, \textit{Theorem} 3.5]{Gor}}.$)$ Let $G$ be a $\pi$-separable group. Then:\par\smallskip
$(1)$ $G\in E_{\{r,\,s\}}$ for every prime $r\in \pi$ and every prime $s\in \pi'$.\par\smallskip
$(2)$ $G\in E_{\pi'\cup\{r\}}$ for every prime $r\in \pi$ and $G\in E_{\pi\cup\{s\}}$ for every prime $s\in \pi'$.\par\smallskip
$(3)$ $G\in E_{\pi}$ and $G\in E_{\pi'}$.
\end{theorem}

Note that the proof of Theorem \ref{h} does not rely on CFSG (i.e. the classification of finite simple groups). In the following, we can give a stronger version of Theorem \ref{h}.

\begin{theorem}\label{hh} Let $G$ be a $\pi$-separable group. Then:\par\smallskip
$(1)$ $G\in D_{\{r,\,s\}}$ for every prime $r\in \pi$ and every prime $s\in \pi'$.\par\smallskip
$(2)$ $G\in D_{\pi'\cup\{r\}}$ for every prime $r\in \pi$ and $G\in D_{\pi\cup\{s\}}$ for every prime $s\in \pi'$.\par\smallskip
$(3)$ $G\in D_{\pi}$ and $G\in D_{\pi'}$.
\end{theorem}

The statement (1) of Theorem \ref{hh} is exactly \cite[Corollary D5.3]{Hal} and the statement (3) follows straightforward from \cite[Theorems D6 and D7]{Hal}. In fact, by \cite[Theorem 7.7]{Rev}, the class $D_\pi$ is closed with respect to taking extensions, equivalently saying, a group $G\in D_\pi$ if and only if every composition factor of $G$ has the $D_\pi$-property. Hence all statements of Theorem \ref{hh} are easy to see.

The converses of the above two theorems certainly have their importance to be investigated. For the sake of simplicity, we now introduce the following definition.

\begin{definition} $(1)$ $U_{\pi,\,\pi'}$ $($resp. $\widehat{U}_{\pi,\,\pi'}$$)$ is the class of all finite groups $G$ such that $G\in E_{\{r,\,s\}}$ $($resp. $G\in D_{\{r,\,s\}}$$)$ for every prime $r\in \pi$ and every prime $s\in \pi'$.\par\smallskip
$(2)$ $V_{\pi,\,\pi'}$ $($resp. $\widehat{V}_{\pi,\,\pi'}$$)$ is the class of all finite groups $G$ such that $G\in E_{\pi'\cup\{r\}}$ $($resp. $G\in D_{\pi'\cup\{r\}}$$)$ for every prime $r\in \pi$ and $G\in E_{\pi\cup\{s\}}$ $($resp. $G\in D_{\pi\cup\{s\}}$$)$ for every prime $s\in \pi'$.\par\smallskip
$(3)$ Following \textup{\cite{Gil1}}, $E_{\pi,\,\pi'}$ $($resp. $D_{\pi,\,\pi'}$$)$ is the class of all finite groups $G$ such that $G\in E_{\pi}$ and $G\in E_{\pi'}$ $($resp. $G\in D_{\pi}$ and $G\in D_{\pi'}$$)$.
\end{definition}

Using these notations, the converse problem can be briefly restated as follows: \textit{what can we say about the structure of the classes $U_{\pi,\,\pi'}$, $\widehat{U}_{\pi,\,\pi'}$, $V_{\pi,\,\pi'}$, $\widehat{V}_{\pi,\,\pi'}$, $E_{\pi,\,\pi'}$ and $D_{\pi,\,\pi'}$}?

We notice that many papers were devoted to study the classes $E_{\pi,\,\pi'}$ and $D_{\pi,\,\pi'}$. 
In \cite{Hal2}, P. Hall proved that a group $G$ is soluble if $G\in E_{p,\,p'}$ for any prime $p$. Later, Z. Arad and M. B. Ward \cite{Ara} generalized the above result by showing that a group $G$ is soluble if $G\in E_{2,\,2'}\cap E_{3,\,3'}$. The complete classification of the class $E_{\pi,\,\pi'}$ was established by Z. Arad and E. Fisman in \cite{Ara1}, and that of the class $D_{\pi,\,\pi'}$ was established by A. L. Gilotti in \cite{Gil1}. Here we quote A. L. Gilotti's result for the completeness (with a rearrangement).\par

\begin{theorem}$($see \textup{\cite[\textit{Theorem} 3.2]{Gil1}}.$)$\label{Gi} \textit{ A group $G\in {D}_{\pi,\,\pi'}$ if and only if for every composition factor $D$ of $G$, one of the following holds:}\par\smallskip
\textit{$(1)$ either $\pi(D)\subseteq \pi$ or $\pi(D)\subseteq \pi'$.}\par\smallskip
\textit{$(2)$ $D$ is isomorphic to $PSL(2,q)$, where $q=3^{f}>3$ with $f\equiv 1\,(mod\,2)$ or $q\equiv 7$ $(mod\, 12)$, and either $\pi\cap \pi(D)=\pi(q+1)$ or $\pi'\cap \pi(D)=\pi(q+1)$.}\par\smallskip
\end{theorem}

As far as we are concerned, only a few results on the classes $U_{\pi,\,\pi'}$, $\widehat{U}_{\pi,\,\pi'}$, $V_{\pi,\,\pi'}$, $\widehat{V}_{\pi,\,\pi'}$ have been obtained. In \cite{Hal}, P. Hall conjectured that a group $G$ is soluble if $G\in U_{p,\,p'}$ (or equivalently, $G\in V_{p,\,p'}$) for any prime $p$, and Z. Arad and M. B. Ward proved this conjecture in \cite{Ara}. Further, the following theorem of V. N. Tyutyanov is surely valuable and useful, which was first announced (without proof) by Z. Du in \cite{Du}.\par

\begin{theorem}$($see \textup{\cite[\textit{Theorem} 1]{Tyu}}.$)$\label{tu} A group $G$ is soluble if $G\in U_{2,\,2'}=V_{2,\,2'}$.
\end{theorem}

In a recent paper \cite{RefJ}, E. P. Vdovin extended Theorem \ref{tu} to give a description of the classes $U_{p,\,p'}$ and $V_{p,\,p'}$ for a fixed prime $p$.
\begin{theorem}$($see \textup{\cite[\textit{Theorem} 9]{RefJ}}.$)$\label{cuo} A group $G$ is $p$-soluble if $G\in U_{p,\,p'}=V_{p,\,p'}$.
\end{theorem}

However, we find the following two examples which illustrate that Theorem \ref{cuo} is not true in general.\par

\begin{example} \textup{Let $G=PSL(2,7)$ and $p=3$. Then by \cite{Alt}, $G\in E_{\{2,\,3\}}$ and $G\in E_{\{3,\,7\}}$, and so $G\in U_{3,\,3'}$. Clearly, $G$ is not $3$-soluble.}
\end{example}

\begin{example} \textup{Let $G=PSU(3,4)$ and $p=3$. Then $G\in E_{\{2,\,3\}}$, $G\in E_{\{3,\,5\}}$ and $G\in E_{\{3,\,13\}}$ by \cite{Alt}. Thus $G\in U_{3,\,3'}$. But $G$ is also not $3$-soluble.}
\end{example}

The main results formulated below can give detailed characterizations of the structure of the classes $U_{\pi,\,\pi'}$, $\widehat{U}_{\pi,\,\pi'}$, $V_{\pi,\,\pi'}$ and $\widehat{V}_{\pi,\,\pi'}$, and we will prove them in Section 2.

\medskip

\noindent\textbf{Theorem A.} \textit{Suppose that a group $G\in U_{\pi,\,\pi'}$. Then $G$ is $\pi'\backslash\{3\}$-soluble when $2\in \pi$ and $G$ is $\pi\backslash\{3\}$-soluble when $2\notin \pi$. More precisely, for every composition factor $D$ of $G$, one of the following holds:}\par\smallskip
\textit{$(1)$ either $\pi(D)\subseteq \pi$ or $\pi(D)\subseteq \pi'$.}\par\smallskip
\textit{$(2)$ $D$ is isomorphic to $PSL(2,7)$ or $PSU(3,q)$, where $q\equiv 4$ or $7$ $(mod\, 9)$, and either $ \pi \cap\pi(D)=\{3\}$ or $ \pi' \cap\pi(D)=\{3\}$.}\par

\medskip

\noindent\textbf{Theorem B.} \textit{$V_{\pi,\,\pi'}\subseteq U_{\pi,\,\pi'}$. Also, if $\pi\neq\{3\}$ and $\pi'\neq\{3\}$, then a group $G$ is $\pi$-separable if and only if $G\in V_{\pi,\,\pi'}$.}

\medskip

\noindent\textbf{Theorem C.} \textit{The following statements are equivalent:}\par\smallskip
\textit{$(1)$ $G$ is $\pi$-separable.}\par\smallskip
\textit{$(2)$ $G\in \widehat{U}_{\pi,\,\pi'}$.}\par\smallskip
\textit{$(3)$ $G\in \widehat{V}_{\pi,\,\pi'}$.}

\medskip

As a corollary of Theorem A, we get that a group $G$ is $p$-soluble if $p\neq 3$ and $G\in U_{p,\,p'}=V_{p,\,p'}$, and $G$ is $\pi$-soluble if $\pi\subseteq \{2,3\}'$ and $G\in U_{\pi,\,\pi'}$ (see below Corollary \ref{cor2}). Note also that the converse of Theorem A does not hold as the next example shows.\par

\begin{example} \textup{Let $G=PGL(2,7)$ and $\pi=\{3\}$. Then for every composition factor $D$ of $G$, either $D\cong \mathbb{Z}_2$ or $D\cong PSL(2,7)$. However, it is easy to see that $G\notin E_{\{2,\,3\}}$. Hence $G\notin U_{3,\,3'}$.}
\end{example}

Moreover, we notice that in \cite{Ara}, it was proved that if $G$ is a simple group of Lie type of characteristic $p$ and $G\in E_{\{p,\,q\}}$ for every odd prime $q\neq p$, then $G$ is isomorphic to $PSL(2,p)$, where $p$ is a Mersenne prime with $p>3$. This result motivates us to study the wider classes defined below.

\begin{definition} $(1)$ $U^*_{\pi,\,\pi'}$ $($resp. $\widehat{U}^*_{\pi,\,\pi'}$$)$ is the class of all finite groups $G$ such that $G\in E_{\{r,\,s\}}$ $($resp. $G\in D_{\{r,\,s\}}$$)$ for every prime $r\in \pi\backslash\{2\}$ and every prime $s\in \pi'\backslash\{2\}$.\par\smallskip
$(2)$ $V^*_{\pi,\,\pi'}$ $($resp. $\widehat{V}^*_{\pi,\,\pi'}$$)$ is the class of all finite groups $G$ such that $G\in E_{\pi'\cup\{r\}}$ $($resp. $G\in D_{\pi'\cup\{r\}}$$)$ for every prime $r\in \pi\backslash\{2\}$ and $G\in E_{\pi\cup\{s\}}$ $($resp. $G\in D_{\pi\cup\{s\}}$$)$ for every prime $s\in \pi'\backslash\{2\}$.
\end{definition}

In Section 3, the structure of the classes $U^*_{\pi,\,\pi'}$, $\widehat{U}^*_{\pi,\,\pi'}$, $V^*_{\pi,\,\pi'}$ and $\widehat{V}^*_{\pi,\,\pi'}$ is investigated and determined.

\section{Proof of Theorems A, B and C}

The next two lemmas in number theory are needed in the proof of Proposition \ref{sim}.\par

\begin{lemma}\label{shu1} $($see \textup{\cite[\textit{Proposition} 2.6(\textit{e})]{Ara1} and \textup{\cite[\textit{Lemma} 2.5]{Ara}}}.$)$ If $p$ is a prime, and $k$ and $n$ are positive integers such that $p^k+1=2^n$ $($resp. $p^k-1=2^n$$)$, then $k=1$ $($resp. either $k=1$ or $k=2$ and $p=n=3$$)$.
\end{lemma}

\begin{lemma}\label{shu} If $k$ and $n$ are positive integers such that $k^2+k+1=3^n$ \textup{(}resp. $k^2-k+1=3^n$\textup{)}, then $k=n=1$ \textup{(}resp. $k=2$ and $n=1$\textup{)}.
\end{lemma}
\begin{proof} Clearly, $k\equiv 1$ $(mod\,3)$ (resp. $k\equiv -1$ $(mod\,3)$). Let $k=3m+1$ (resp. $k=3m+2$). Then $9m^2+9m+3=3^n$, and so $3m^2+3m+1=3^{n-1}$. This implies that $m=0$ and $n=1$. Hence $k=n=1$ (resp. $k=2$ and $n=1$).
\end{proof}

Let $r$ be an odd prime and $q$ be an integer with $(q,r)=1$. We denote by $e(q, r)$
the least positive integer $e$ such that $q^e\equiv 1$ $(mod\,r)$. Also, if $q$ is an odd integer, then we set
$e(q, 2) = 1$ if $q\equiv 1$ $(mod\,4)$, and $e(q, 2) = 2$ if $q\equiv 3$ $(mod\,4)$. Now we can establish the following proposition which plays a fundamental role in our whole argument.\par

\begin{proposition}\label{sim} Let $S$ be a non-abelian simple group and $\pi$ be a non-empty set of primes properly contained in $\pi(S)\backslash\{2\}$. Then $S\in U^*_{\pi,\,\pi'}$ if and only if one of the following holds:\par\smallskip
$(1)$ $S$ is isomorphic to $PSL(2,p)$, where $p$ is a Mersenne prime with $p>3$.\par\smallskip
$(2)$ $S$ is isomorphic to $PSL(3,q)$, where $q=2^{f}>2$ with $f\equiv \pm 1\,(mod\,6)$, and either $\pi=\{3\}$ or $\pi=\pi(S)\backslash \{2,3\}$.\par\smallskip
$(3)$ $S$ is isomorphic to $PSU(3,q)$, where $q\equiv 4$ or $7$ $(mod\, 9)$, and either $\pi=\{3\}$ or $\pi=\pi(S)\backslash \{2,3\}$.\par
\end{proposition}
\begin{proof} The proof of the necessity part is based on CFSG.
Note that by \cite[Theorem A4]{Hal}, $S$ can not be isomorphic to an alternating group $A_n$ of degree $n\geq 5$. Now suppose that $S$ is isomorphic to a sporadic simple group or a Tit group $^2F_4(2)'$. Then \cite[Corollary 6.13]{Gro} can easily yield a contradiction. Hence according to CFSG, $S$ is isomorphic to a simple group of Lie
type with a ground field $\mathbb{F}_q$ of characteristic $p$. If either $\pi=\{p\}$ or $\pi'\cap \pi(S)=\{2,p\}$, then $p>2$. By \cite[Proposition 3.2]{Ara}, $S$ is isomorphic to $A_1(p)\cong PSL(2,p)$, where $p$ is a Mersenne prime with $p>3$. We may, therefore, assume that $\pi\neq \{p\}$ and $\pi'\cap \pi(S)\neq \{2,p\}$, and proceed the proof via the following steps.\par

\smallskip
(1) \textit{$S$ can not be isomorphic to a Suzuki group or a Ree group.}\par
\smallskip

If $S$ is isomorphic to a Suzuki group $^2B_2(q)$, where $q=2^{2m+1}$, then $|S|=q^2(q^2+1)(q-1)$. By \cite[Lemma 14]{Vdo2}, for every prime $r\in \pi$ and every prime $s\in (\pi'\cap \pi(S)) \backslash\{2\}$, either $\{r,s\}\subseteq\pi(q-1)$ or $\{r,s\}\subseteq\pi(q\pm \sqrt{2q}+1)$. This does not occur because $(q-1,q^2+1)=1$. Thus $S\not \cong$ $^2B_2(q)$.\par

Now assume that $S$ is isomorphic to a Ree group $^2G_2(q)$, where $q=3^{2m+1}$. Then $|S|=q^3(q^3+1)(q-1)$. By \cite[Lemma 14]{Vdo2}, for every prime $r\in \pi\backslash\{3\}$ and every prime $s\in (\pi'\cap \pi(S)) \backslash\{2,3\}$, either $\{r,s\}\subseteq\pi(q\pm1)$ or $\{r,s\}\subseteq\pi(q\pm \sqrt{3q}+1)$. If $q+1$ is a power of $2$, then $q=3$ by Lemma \ref{shu1}, which is impossible. Thus there exists a prime $k$ with $k\in \pi(q+1)\backslash\{2\}$. Without loss of generality, we may assume that $k\in \pi'\cap \pi(S)$. Then $r\in \pi(q+1)\backslash\{2\}$ for every prime $r\in \pi\backslash\{3\}$. Since $(q+1, q+\sqrt{3q}+1)=1$, we have that $q+\sqrt{3q}+1$ is a power of $2$, a contradiction. Hence $S\not \cong$ $^2G_2(q)$.\par

Finally, suppose that $S$ is isomorphic to a Ree group $^2F_4(q)$, where $q=2^{2m+1}$. Then $|S|=q^{12}(q^6+1)(q^4-1)(q^3+1)(q-1)$. By \cite[Lemma 14]{Vdo2}, for every prime $r\in \pi$ and every prime $s\in (\pi'\cap \pi(S)) \backslash\{2\}$, $\{r,s\}\subseteq\pi(q^2\pm1)$ or $\{r,s\}\subseteq\pi(q\pm \sqrt{2q}+1)$ or $\{r,s\}\subseteq\pi(q^2\pm q\sqrt{2q}\mp \sqrt{2q}-1)$ or $\{r,s\}\subseteq\pi(q^2\pm q\sqrt{2q}+q\pm \sqrt{2q}+1)$. This does not occur because $(q^2-1,q^2+1)=1$. Thus $S\not \cong$ $^2F_4(q)$.\par

\smallskip
$(2)$ \textit{$S$ can not be isomorphic to one of the following groups: $B_{n}(q)$ $(n>1)$; $C_{n}(q)$ $(n>2)$; $D_{n}(q)$ $(n>3)$; $^2D_{n}(q)$ $(n>3)$; $^3D_4(q)$; $E_6(q)$; $^2E_6(q)$; $E_7(q)$; $E_8(q)$; $F_4(q)$; $G_2(q)$ $(q\geq 3)$.}\par
\smallskip

Suppose that the statement $(2)$ does not hold. Then $q^2-1$ divides $|S|$.
Note that by Lemma \ref{shu1}, $q^2-1$ is not a power of $2$ except $q=3$. If $q\neq 3$, then there exists a prime $k\in \pi(q^2-1)\backslash\{2\}$. Without loss of generality, we may assume that $k\in \pi'\cap \pi(S)$. Then according to \cite[Table 7]{Vdo}, for every prime $r\in \pi\backslash \{p\}$, $e(q,r)=e(q,k)=1$ or $2$. Thus by \cite[Table 7]{Vdo} again, for every prime $t\in \pi(S)\backslash \{2,p\}$, $e(q,t)=e(q,k)=1$ or $2$. This implies that either $\pi(S)\backslash \{2,p\}\subseteq \pi(q-1)$ and $q+1$ is a power of $2$, or $\pi(S)\backslash \{2,p\}\subseteq \pi(q+1)$ and $q-1$ is a power of $2$. If $q^2+1$ divides $|S|$, then $q^2+1$ is a power of $2$ because $(q^2+1,q^2-1)$ divides $2$. This does not occur by Lemma \ref{shu1}. Therefore, it is easy to see that $q^4+q^2+1$ divides $|S|$. Since $(q^4+q^2+1, q^2-1)$ divides $3$, we have that $q^4+q^2+1$ is a power of $3$, and so $q^2+q+1$ and $q^2-q+1$ are both powers of $3$, which is impossible.\par

Now assume that $q=3$. Then clearly, $13\in \pi(S)$ except $S\cong B_2(3)\cong PSU(4,2)$. If $S\cong PSU(4,2)$, then $\pi(S)=\{2,3,5\}$. However, $S\notin E_{\{3,\,5\}}$ by \cite{Alt}, which is impossible. Thus $S\not\cong PSU(4,2)$, and so $13\in \pi(S)$. We may, then, assume that $13\in \pi'$.
If $S\cong$ $^2D_{6}(3)$, then $5\in \pi(S)$. By \cite[Table 7]{Vdo}, for every prime $r\in \pi\backslash\{3\}$, $e(3,r)=3$ or $6$ because $e(3,13)=3$. Since $\pi(3^6-1)=\{2,7,13\}$, it is clear that $7\in \pi$. However, $S\notin E_{\{5,\,7\}}$ and $S\notin E_{\{5,\,13\}}$ by \cite[Table 7]{Vdo}. This contradiction forces that $S\not\cong$ $^2D_{6}(3)$. Then by \cite[Table 7]{Vdo}, $e(3,r)=e(3,13)=3$ for every prime $r\in \pi\backslash\{3\}$. This does not occur because $\pi(3^3-1)=\{2,13\}$. Therefore, the statement $(2)$ follows. \par

\smallskip
(3) \textit{If $S$ is isomorphic to $A_{n-1}(q)\cong PSL(n,q)$ $(n>1)$, then one of the following holds: $(a)$ $S$ is isomorphic to $PSL(2,p)$, where $p$ is a Mersenne prime with $p>3$; $(b)$ $S$ is isomorphic to $PSL(3,q)$, where $q=2^{f}$ with $f\equiv \pm 1\,(mod\,6)$, and either $\pi=\{3\}$ or $\pi=\pi(S)\backslash \{2,3\}$.}\par
\smallskip

Without loss of generality, we assume that $p\in \pi$ when $p>2$. Since $\pi'\cap \pi(S)\neq \{2\}$, there exists a prime $k\in (\pi'\cap \pi(S))\backslash\{2\}$.
Firstly, suppose that $e(q,k)>1$ and $p>2$. Then by \cite[Theorem 5.5]{Spi}, we have that $q=3$, and so $13\in \pi(S)$. Note that $\pi(3^3-1)=\{2,13\}$. If $13\in \pi$, then $e(3,k)=3$ by \cite[Table 7]{Vdo}, which is impossible. Thus $13\in \pi'$. By \cite[Table 7]{Vdo} again, $e(3,r)=3$ for every prime $r\in \pi\backslash \{3\}$, which is also impossible. Therefore, this case does not occur.\par

Now assume that $e(q,k)=1$. By \cite[Table 7]{Vdo}, for every prime $r\in \pi\backslash \{p\}$, one of the following holds: (1) $e(q,r)=1$; (2) $e(q,r)=r-1$, $(q^{r-1}-1)_r=r$ and $[\frac{n}{r-1}]=[\frac{n}{r}]$. Then we discuss two possible cases below:\par

\smallskip
(i) \textit{Case $1$: $n=2$.}\par
\smallskip

In this case, if $q+1$ is not a power of $2$, then there exists a prime $t\in \pi(S)\backslash \{2,p\}$ with $e(q,t)=2$. If $t\in \pi$, then the above argument shows that $t=3$, and so $n=3$ because $[\frac{n}{2}]=[\frac{n}{3}]$. This contradiction forces that
$t\in \pi'$. Hence $l\in \pi'$ for every prime $l\in \pi(S)\backslash \{2,p\}$ with $e(q,l)=2$. In view of \cite[Table 7]{Vdo}, we have that $t=3$ and $[\frac{n}{2}]=[\frac{n}{3}]$. It follows that $n=3$, against supposition. Thus $q+1$ is a power of $2$. Then by Lemma \ref{shu1}, $q=p$, and obviously, $p$ is a Mersenne prime with $p>3$. Hence the statement (a) follows.\par

\smallskip
(ii) \textit{Case $2$: $n>2$.}\par
\smallskip

In this case, $q^2+q+1$ divides $|S|$. If there exists no prime $t\in \pi(S)\backslash \{2,p\}$ with $e(q,t)=3$, then since $(q^2+q+1, q-1)$ divides $3$, $q^2+q+1$ is a power of $3$. This does not occur by Lemma \ref{shu}. We may, therefore, assume that there is a prime $t\in \pi(S)\backslash \{2,p\}$ with $e(q,t)=3$. Then $t\in \pi'$, and the above discussion yields that $p=2$. Also, it follows from \cite[Table 7]{Vdo} that $\pi=\{3\}$ and $e(q,3)=2$. Hence $(q+1)_3=3$ and $[\frac{n}{2}]=[\frac{n}{3}]$, and so $n=3$. Let $q=2^f$. Since $3\mid 2^f+1$ and $9\nmid 2^f+1$, it is easy to see that $f\equiv \pm 1\,(mod\,6)$. Thus the statement (b) holds.\par

Finally, suppose that $e(q,s)>1$ for every prime $s\in (\pi'\cap \pi(S))\backslash\{2\}$ and $p=2$. Since $2\notin \pi(q-1)$, there is a prime $t\in \pi(S)\backslash \{2\}$ with $e(q,t)=1$. Also, for every prime $l\in \pi(S)\backslash \{2\}$ with $e(q,l)=1$, we have that $l\in \pi$. Then by \cite[Table 7]{Vdo}, for every prime $s\in (\pi'\cap \pi(S))\backslash\{2\}$, $e(q,s)=s-1$, $(q^{s-1}-1)_s=s$ and $[\frac{n}{s-1}]=[\frac{n}{s}]$. With a similar argument as above, $S\cong PSL(3,q)$, where $q=2^{f}$ with $f\equiv \pm 1\,(mod\,6)$, and $\pi=\pi(S)\backslash \{2,3\}$. This shows that the statement (b) follows, and thus the proof of the statement (3) is complete.\par

\smallskip
$(4)$ \textit{If $S$ is isomorphic to $^2A_{n-1}(q)\cong PSU(n,q)$ $(n>2)$, then $S$ is isomorphic to $PSU(3,q)$, where $q\equiv 4$ or $7$ $(mod\, 9)$, and either $\pi=\{3\}$ or $\pi=\pi(S)\backslash \{2,3\}$.}\par
\smallskip

Without loss of generality, we assume that $p\in \pi$ when $p>2$. Let $k$ be a prime with $k\in (\pi'\cap \pi(S))\backslash\{2\}$.
Firstly, suppose that $e(q,k)>1$ and $p>2$. Then by \cite[Theorem 8.3]{Vdo}, a Hall $\{p,k\}$-subgroup of $S$ is contained in a Borel subgroup or is a parabolic subgroup of $S$. Hence by \cite[Theorems 8.4 and 8.7]{Vdo}, we have that $k\in \pi(q-1)$, a contradiction occurs.\par

Now assume that $e(q,k)>1$ and $p=2$. Then we discuss two possible cases in the following:\par

\smallskip
(i) \textit{Case $1$: $\pi(q-1)\nsubseteq \pi'$.}\par
\smallskip

In this case, there exists a prime $t\in \pi$ with $e(q,t)=1$. Then by \cite[Table 7]{Vdo}, the following hold: (1) $t=3$; (2) for every prime $s\in (\pi'\cap \pi(S))\backslash\{2\}$, $e(q,s)=1$ or $2$ or $6$; (3) $(q-1)_3=3$ and $\pi(\frac{q-1}{3})\subseteq \pi'\cap \pi(S)$; (4) either $[\frac{n}{2}]=[\frac{n}{3}]$ or $[\frac{n}{2}]=[\frac{n}{3}]+1$ and $n\equiv -1$ $(mod\, 3)$. This implies that $q\equiv 4$ or $7$ $(mod\, 9)$ and $n=3$ or $5$.\par

We shall prove that $n=3$ and $\pi=\{3\}$. Suppose that $q=4$. Since $\pi(4^6-1)=\{3,5,7,13\}$ and $e(4,7)=3$, $\pi'\cap \pi(S)\subseteq \{2,5,13\}$. If $S\cong PSU(5,4)$, then $17\in \pi$. However, $S\notin E_{\{5,\,17\}}$ and $S\notin E_{\{13,\,17\}}$ by \cite[Table 7]{Vdo}. This contradiction implies that $S\cong PSU(3,4)$, and so $\pi(S)=\{2,3,5,13\}$. Since $S\notin E_{\{5,\,13\}}$ by \cite{Alt}, $\pi=\{3\}$. We may, therefore, assume that $q>4$. Then $\pi(\frac{q-1}{3})\neq \varnothing$. Hence there is a prime $l\in (\pi'\cap \pi(S))\backslash\{2\}$ with $e(q,l)=1$. By \cite[Table 7]{Vdo}, $e(q,r)=1$ for every prime $r\in \pi$. This implies that $\pi=\{3\}$. Therefore, for every prime $t\in \pi(S)\backslash\{2\}$, $e(q,t)=1$ or $2$ or $6$. If $n=5$, then $q^2+1$ divides $|S|$, and so $q^2+1$ is a power of $2$ because $(q^2+1,q^2-1)=1$, which is impossible. Thus $n=3$.\par

\smallskip
(ii) \textit{Case $2$: $\pi(q-1)\subseteq \pi'$.}\par
\smallskip

In this case, $e(q,r)>1$ for every prime $r\in \pi$. Then by \cite[Table 7]{Vdo}, the following hold: (1) $3\in \pi'$; (2) $q=4$; (3) $e(4,r)=2$ or $6$ for every prime $r\in \pi$; (4) either $[\frac{n}{2}]=[\frac{n}{3}]$ or $[\frac{n}{2}]=[\frac{n}{3}]+1$ and $n\equiv -1$ $(mod\, 3)$. By discussing similarly as above, $S\cong PSU(3,4)$ and $\pi'\cap \pi(S)=\{2,3\}$. However, this case does not occur because $e(4,k)>1$.\par

Finally, suppose that $e(q,s)=1$ for every prime $s\in (\pi'\cap \pi(S))\backslash\{2\}$. If $\pi(S)\backslash\{2,p\}\subseteq\pi(q-1)$, then $q^2-q+1$ is a power of $2$ because $(q^2-q+1,q-1)=1$. This contradiction forces that $\pi(S)\backslash\{2,p\}\nsubseteq\pi(q-1)$, and so there is a prime $t\in \pi$ with $e(p,t)>1$. Then by \cite[Table 7]{Vdo}, the following hold: (1) $\pi'\cap \pi(S)=\{2,3\}$; (2) for every prime $r\in \pi\backslash\{p\}$, $e(q,r)=1$ or $2$ or $6$; (3) $(q-1)_3=3$; (4) either $[\frac{n}{2}]=[\frac{n}{3}]$ or $[\frac{n}{2}]=[\frac{n}{3}]+1$ and $n\equiv -1$ $(mod\, 3)$. This yields that $q\equiv 4$ or $7$ $(mod\, 9)$ and $n=3$ or $5$. If $n=5$, then $q^2+1$ divides $|S|$. Since $(q^2+1,q^2-1)$ divides $2$, $q^2+1$ is a power of $2$, which contradicts Lemma \ref{shu1}. Hence $n=3$. Considering together, the statement (4) holds. The necessity part is thus proved, due to CFSG.\par\smallskip

Conversely, according to \cite[Theorem 5.5]{Spi}, \cite[Table 7]{Vdo} and \cite[Theorems 8.3 and 8.4]{Vdo}, the sufficiency part can be easily verified.\end{proof}

The next lemma is a collection of properties of the $E_\pi$, $C_\pi$ and $D_\pi$-property, and will be used heavily in the sequel.\par

\begin{lemma}\label{hal} Let $N$ be a normal subgroup of $G$. Then:\par\smallskip

$(1)$ \textup{(\textit{see} \cite[\textit{Lemma} 1]{Hal}.)} If $G\in E_{\pi}$, then $N\in E_{\pi}$ and $G/N\in E_{\pi}$.\par\smallskip
$(2)$ \textup{(\textit{see} \cite[\textit{Theorem} E2]{Hal}.)} If $N\in C_{\pi}$ and $G/N\in E_{\pi}$, then $G\in E_{\pi}$.\par\smallskip
$(3)$ \textup{(\textit{see} \cite[\textit{Theorem} 7.7]{Rev}.)} $G\in D_{\pi}$ if and only if $N\in D_{\pi}$ and $G/N\in D_{\pi}$.
\end{lemma}

Now we are ready to prove Theorems A, B and C, and present several applicable corollaries.\par

\begin{proof}[\textup{\textbf{Proof of Theorem A}}]

By Lemma \ref{hal}(1), $D\in U_{\pi,\,\pi'}$ for every composition factor $D$ of $G$. Without loss of generality, we assume that $2\notin \pi$, $\pi(D)\nsubseteq \pi$ and $\pi(D)\nsubseteq \pi'$. If $\pi\cap \pi(D)=\pi(D)\backslash\{2\}$, then $D\in U_{2,\,2'}$, and so $D$ is soluble by Theorem \ref{tu}. This contradiction shows that $\pi\cap \pi(D)$ is a non-empty set of primes properly contained in $\pi(D)\backslash\{2\}$. In view of Proposition \ref{sim}, we shall discuss three possible cases below.\par

Firstly, suppose that $D\cong PSL(2,p)$, where $p$ is a Mersenne prime with $p>3$. Then $|D|_2=p+1$. Let $H$ be a Hall subgroup of $D$ having order divisible by $2$. With a glance at L.  E. Dickson's list of subgroups of $PSL(2,p)$ (for example, see \cite[Chap. II, Hauptsatz 8.27]{Hup}), either $H$ is a Sylow $2$-subgroup of $D$ or $D\cong PSL(2,7)$ and $H$ is a Hall $\{2,3\}$-subgroup of $D$. This implies that $D\cong PSL(2,7)$ and $\pi\cap \pi(D)=\{3\}$. Now assume that $D\cong PSL(3,q)$, where $q=2^{f}>2$ with $f\equiv \pm 1\,(mod\,6)$, and either $\pi\cap \pi(D)=\{3\}$ or $\pi'\cap \pi(D)=\{2,3\}$. Then clearly, $(q+1)_3=3$.
By \cite[Theorem 5.5]{Spi}, $D\notin E_{\{2,\,3\}}$. This implies that $\pi'\cap \pi(D)=\{2,3\}$ and $\pi(D)\backslash \{2,3\}\subseteq \pi(q-1)$ by \cite[Theorem 5.5]{Spi} again. Since $(q+1,q-1)=1$, we have that $q+1=3$, which is absurd. Finally, suppose that $D\cong PSU(3,q)$, where $q\equiv 4$ or $7$ $(mod\, 9)$, and either $\pi\cap \pi(D)=\{3\}$ or $\pi'\cap \pi(D)=\{2,3\}$. If $\pi'\cap \pi(D)=\{2,3\}$, then by \cite[Theorem 5.2]{Rev}, $\pi(D)\backslash \{2,3,p\}\subseteq \pi(q^2-1)$. Since $(q^2-q+1,q^2-1)$ divides $3$, $q^2-q+1$ is a power of $3$, and so $q=2$ by Lemma \ref{shu}, a contradiction. Thus $\pi\cap \pi(D)=\{3\}$. Considering together, one of the statements (1) and (2) follows. Also, it is easy to see that $G$ is $\pi\backslash\{3\}$-soluble by the Feit-Thompson theorem. This finishes the proof of Theorem A.\end{proof}

A simple group $G$ is called a simple $K_n$-group if $|\pi(G)|=n$. In \cite{Her} and \cite{Bug}, M. Herzog and Y. Bugeaud et al. gave detailed classifications of simple $K_3$-groups and simple $K_4$-groups, respectively. For reader's convenience, we quote M. Herzog's result below.\par

\begin{lemma}\textup{(\textit{see} \cite{Her}.)}\label{he} If $G$ is a simple $K_3$-group, then $G$ is isomorphic to one of the simple groups: $A_5$, $A_6$, $PSL(2,7)$, $PSL(2,8)$, $PSL(2,17)$, $PSL(3,3)$, $PSU(3,3)$, $PSU(4,2)$.
\end{lemma}

\begin{corollary}\label{cor2} Suppose that  $3\notin \pi$ and one of the following holds:\par\smallskip
$(1)$ $2\notin \pi$.\par\smallskip
$(2)$ $|\pi|=1$.\par\smallskip
$(3)$ $|\pi|=2$ and $\pi\neq\{2,7\}$.\par\smallskip
$(4)$ $|\pi|=3$, $\{2,7\}\nsubseteq \pi$ and $\pi\neq\{2,5,13\}$.\par\smallskip
\noindent Then a group $G$ is $\pi$-soluble if and only if $G\in U_{\pi,\,\pi'}$.\par
\end{corollary}
\begin{proof} The necessity part is obvious by Theorem \ref{h}, and we only need to prove the sufficiency part. If the statement (1) or (2) holds, then this corollary follows immediately from Theorem A and Theorem \ref{tu}. Next, suppose that the statement (3) holds. If $G$ has a non-abelian composition factor $D$ with $\pi(D)\nsubseteq\pi'$, then by Theorem A and Burnside's $p^aq^b$-theorem, $\pi(D)=\pi\cup\{3\}$. It follows from Theorem A and Lemma \ref{he} that $D\cong PSL(2,7)$. Thus $\pi=\{2,7\}$. This contradiction yields that $G$ is $\pi$-soluble. Now assume that the statement (4) follows. If $G$ has a non-abelian composition factor $D$ with $\pi(D)\nsubseteq\pi'$, then by Theorem A, $\pi(D)\subseteq\pi\cup\{3\}$, and so $|\pi(D)|\leq 4$. Hence by Theorem A, Lemma \ref{he} and \cite[Theorem 1]{Bug}, we have that $D\cong PSL(2,7)$ or $PSU(3,4)$ or $PSU(3,7)$. This indicates that either $\{2,7\}\subseteq \pi$ or $\pi=\{2,5,13\}$, neither is possible. Therefore, $G$ is $\pi$-soluble.\end{proof}


The following lemma can be viewed as a critical step in the proof of Theorem B.\par

\begin{lemma}\label{fei} Suppose that $2\notin \pi$ and a non-abelian simple group $S\in V^*_{\pi,\,\pi'}$. Then $|\pi\cap \pi(S)|\leq 1$.
\end{lemma}

\begin{proof} We may assume that $\pi\cap \pi(S)\neq \varnothing$. If $\pi'\cap \pi(S)=\{2\}$, then by Theorem \ref{tu}, $S$ is soluble, a contradiction. Hence there exists a prime $s\in (\pi'\cap \pi(S))\backslash\{2\}$. Let $A$ be a Hall $\pi\cup\{s\}$-subgroup of $S$ and $N$ be a minimal normal subgroup of $A$. Since $A$ is soluble by the Feit-Thompson theorem, $N$ is abelian, and so $S$ has a Hall $\pi'\cup\{r\}$-subgroup $B$ containing $N$, where $r\in \pi\cap \pi(S)$. This implies that $S=\langle N^S\rangle=\langle N^B\rangle\leq B$. Thus $S=B$, and certainly $\pi\cap \pi(S)=\{r\}$. Therefore, $|\pi\cap \pi(S)|\leq 1$ as wanted.
\end{proof}

\begin{proof}[\textup{\textbf{Proof of Theorem B}}] Firstly, suppose that $\pi\neq \{3\}$ and $\pi'\neq \{3\}$. If $G\in V_{\pi,\,\pi'}$ and $G$ is not $\pi$-separable, then $G$ has a composition factor $D$ with $\pi(D)\nsubseteq \pi$ and $\pi(D)\nsubseteq \pi'$. Since $D\in V_{\pi,\,\pi'}$ by Lemma \ref{hal}(1), either $|\pi\cap \pi(D)|=1$ or $|\pi'\cap \pi(D)|=1$ by Lemma \ref{fei}. Then it is easy to see that $D\in U_{\pi,\,\pi'}$. It follows from Theorem A that $D\cong PSL(2,7)$ or $PSU(3,q)$, where $q\equiv 4$ or $7$ $(mod\, 9)$, and either $ \pi \cap\pi(D)=\{3\}$ or $ \pi' \cap\pi(D)=\{3\}$. In both cases, we have that $D\in E_{3'}$ because $\pi\neq \{3\}$ and $\pi'\neq \{3\}$. Hence by \cite[Corollary 5.6]{Ara1}, $D\cong PSL(2,8)$, a contradiction occurs. Thus $G$ is $\pi$-separable if $G\in V_{\pi,\,\pi'}$. Also by Theorem \ref{h}, $G\in V_{\pi,\,\pi'}$ if $G$ is $\pi$-separable. This shows that $G$ is $\pi$-separable if and only if $G\in V_{\pi,\,\pi'}$. We may, therefore, assume that either $\pi=\{3\}$ or $\pi'=\{3\}$. Then clearly, $G\in U_{3,\,3'}$, which completes the proof of Theorem B.\end{proof}

Now we can easily reproduce a result due to Z. Du.\par

\begin{corollary}\label{du1}$($see \textup{\cite[\textit{Theorems }1 \textit{and} 3]{Du}}.$)$ The following statements are equivalent:\par\smallskip
\textit{$(1)$ $G$ is $\pi$-separable.}\par\smallskip
\textit{$(2)$ $G\in U_{\pi,\,\pi'}\cap E_{\pi,\,\pi'}$.}\par\smallskip
\textit{$(3)$ $G\in V_{\pi,\,\pi'}\cap E_{\pi,\,\pi'}$.}
\end{corollary}

\begin{proof} In view of Theorem \ref{h} and Theorem B, (1) implies (3) and (3) implies (2). So we only need to show that (2) implies (1). Since $G\in U_{\pi,\,\pi'}\cap E_{\pi,\,\pi'}$, by Lemma \ref{hal}(1), $D\in U_{\pi,\,\pi'}\cap E_{\pi,\,\pi'}$ for every composition factor $D$ of $G$. If $D\in E_{3'}$, then $D\cong PSL(2,8)$ by \cite[Corollary 5.6]{Ara1}. This yields from Theorem A that either $\pi(D)\subseteq \pi$ or $\pi(D)\subseteq \pi'$, and so $G$ is $\pi$-separable. The proof is thus ended.
\end{proof}



\begin{proof}[\textup{\textbf{Proof of Theorem C}}] Obviously, (1) implies (3) by Theorem \ref{hh}. Now suppose that $G\in \widehat{V}_{\pi,\,\pi'}$. Then by Theorem B, $G$ is $\pi$-separable unless either $\pi=\{3\}$ or $\pi'=\{3\}$. No matter what happens, we have that $G\in \widehat{U}_{\pi,\,\pi'}$ by Theorem \ref{hh}. Thus (3) implies (2). Finally, we shall prove that (2) implies (1). By Theorem A, if $G\in \widehat{U}_{\pi,\,\pi'}$ and $G$ is not $\pi$-separable, then $G$ has a composition factor $D$ which is isomorphic to $PSL(2,7)$ or $PSU(3,q)$, where $q\equiv 4$ or $7$ $(mod\, 9)$, and either $\pi\cap \pi(D)=\{3\}$ or $\pi'\cap \pi(D)=\{3\}$. This implies that $D\in D_{\{2,\,3\}}$ by Lemma \ref{hal}(3), which contradicts \cite[Lemma 6.1]{Rev}. Hence $G$ is $\pi$-separable.
\end{proof}

\section{Further investigations on Hall subgroups}

In this section, we concentrate our attention on the classes $U^*_{\pi,\,\pi'}$, $\widehat{U}^*_{\pi,\,\pi'}$, $V^*_{\pi,\,\pi'}$ and $\widehat{V}^*_{\pi,\,\pi'}$, and obtain Theorems \ref{E3}, \ref{D2}, \ref{E4} and \ref{D3} to give complete classifications of these classes, respectively, which can be viewed as analogs of Theorems A, B and C. We also arrive at a number of corollaries from the above-mentioned theorems.

\begin{theorem}\label{E3} \textit{A group $G\in U^*_{\pi,\,\pi'}$ if and only if for every composition factor $D$ of $G$, one of the following holds:}\par\smallskip
\textit{$(1)$ either $\pi\cap \pi(D)\subseteq \{2\}$ or $\pi'\cap \pi(D)\subseteq \{2\}$.}\par\smallskip
\textit{$(2)$ $D$ is isomorphic to $PSL(2,p)$, where $p$ is a Mersenne prime with $p>3$.}\par\smallskip
\textit{$(3)$ $D$ is isomorphic to $PSL(3,q)$, where $q=2^{f}>2$ with $f\equiv \pm 1\,(mod\,6)$, and either $ \pi\cap \pi(D)\subseteq \{2,3\}$ or $\pi'\cap \pi(D)\subseteq \{2,3\}$.}\par\smallskip
\textit{$(4)$ $D$ is isomorphic to $PSU(3,q)$, where $q\equiv 4$ or $7$ $(mod\, 9)$, and either $ \pi\cap \pi(D)\subseteq \{2,3\}$ or $\pi'\cap \pi(D)\subseteq \{2,3\}$.}\par
\end{theorem}

\begin{proof}For the necessity part, by Lemma \ref{hal}(1), we have that $D\in U^*_{\pi,\,\pi'}$ for every composition factor $D$ of $G$. If either $(\pi\cap \pi(D))\backslash\{2\}=\varnothing$ or $(\pi\cap \pi(D))\backslash\{2\}=\pi(D)\backslash\{2\}$, then the statement (1) holds trivially. Hence we may assume that $(\pi\cap \pi(D))\backslash\{2\}$ is a non-empty set of primes properly contained in $\pi(D)\backslash\{2\}$. Then the necessity part follows straightforward from Proposition \ref{sim}.
Now we proceed to prove the sufficiency part. By Proposition \ref{sim} and F. Gross's well-known result (see \cite[Theorem A]{Gro3}), for every composition factor $D$ of $G$, $D \in C_{\{r,s\}}$ for every prime $r\in \pi\backslash\{2\}$ and every prime $s\in \pi'\backslash\{2\}$. Hence $G\in U^*_{\pi,\,\pi'}$ by Lemma \ref{hal}(2). This shows that the sufficiency part is true.
\end{proof}

\begin{corollary}\label{cor1} If $r\notin \{2,3\}$ and a group $G\in U^*_{r,\,r'}$, then either $G$ is $r$-soluble or $G$ has a composition factor isomorphic to $PSL(2,p)$, where $p$ is a Mersenne prime with $r\in \pi(\frac{p(p-1)}{2})$.
\end{corollary}
\begin{proof} By Theorem \ref{E3} and Burnside's $p^aq^b$-theorem, for every composition factor $D$ of $G$, $\pi(D)=\{r\}$ or $\pi(D)\subseteq r'$ or $\pi(D)\subseteq \{2,3,r\}$ or $D\cong PSL(2,p)$, where $p$ is a Mersenne prime with $r\in \pi(\frac{(p(p-1)}{2})$. If $\pi(D)\subseteq \{2,3,r\}$, then by Theorem \ref{E3} and Lemma \ref{he}, we have that $D\cong PSL(2,7)$. Therefore, either $G$ is $r$-soluble or $G$ has a composition factor isomorphic to $PSL(2,p)$, where $p$ is a Mersenne prime with $r\in \pi(\frac{p(p-1)}{2})$.
\end{proof}

\begin{theorem}\label{D2} \textit{A group $G\in \widehat{U}^*_{\pi,\,\pi'}$ if and only if for every composition factor $D$ of $G$, one of the following holds:}\par\smallskip
\textit{$(1)$ either $\pi\cap \pi(D)\subseteq \{2\}$ or $\pi'\cap \pi(D)\subseteq \{2\}$.}\par\smallskip
\textit{$(2)$ $D$ is isomorphic to $PSL(2,p)$, where $p$ is a Mersenne prime with $p>3$.}\par\smallskip
\textit{$(3)$ $D$ is isomorphic to $PSU(3,p)$, where $p=7$ or $p=2^f-1$ is a Mersenne prime with $f\equiv 5$ $(mod\, 6)$, and either $ \pi\cap \pi(D)\subseteq \{2,3\}$ or $\pi'\cap \pi(D)\subseteq \{2,3\}$.}\par
\end{theorem}

\begin{proof} By Lemma \ref{hal}(3), we only need to prove that $D\in \widehat{U}^*_{\pi,\,\pi'}$ if and only if one of statements (1)-(3) holds. If one of statements (1)-(3) holds, then by \cite[Theorem 2]{Rev1}, $D\in \widehat{U}^*_{\pi,\,\pi'}$. We may then, suppose that $D\in \widehat{U}^*_{\pi,\,\pi'}$, $\pi\cap \pi(D)\nsubseteq \{2\}$ and $\pi'\cap \pi(D)\nsubseteq \{2\}$. Then by Theorem \ref{E3}, we should consider three possible cases.\par
If $D\cong PSL(2,p)$, where $p$ is a Mersenne prime with $p>3$, then no additional condition is required. Next, suppose that $D\cong PSL(3,q)$, where $q=2^{f}>2$ with $f\equiv \pm 1\,(mod\,6)$, and either $ \pi\cap \pi(D)\subseteq \{2,3\}$ or $\pi'\cap \pi(D)\subseteq \{2,3\}$. Then $D\in \widehat{U}^*_{3,\,3'}$. Since $e(q,3)=2$, we have $e(q,k)=2$ or $3$ for every prime $k\in \pi(D)\backslash\{2\}$ by \cite[Theorem 2]{Rev1}. This forces that $q-1=1$, and so $q=2$, a contradiction. Now assume that $D$ is isomorphic to $PSU(3,q)$, where $q\equiv 4$ or $7$ $(mod\, 9)$, and either $ \pi\cap \pi(D)\subseteq \{2,3\}$ or $\pi'\cap \pi(D)\subseteq \{2,3\}$. Then also $D\in \widehat{U}^*_{3,\,3'}$. Since $e(q,3)=1$, $e(q,k)=1$ or $6$ for every prime $k\in \pi(D)\backslash\{2,p\}$ by \cite[Theorem 2]{Rev1}. Note that $|D|=q^3(q^3+1)(q^2-1)$. From this we can deduce that $q+1$ is a power of $2$. Then by Lemma \ref{shu1}, $q=p$ is a Mersenne prime. Let $p=2^f-1$ with a prime number $f$. Since $p\equiv 4$ or $7$ $(mod\, 9)$, it is easy to obtain that either $f=3$ or $f\equiv 5$ $(mod\, 6)$. Thus one of statements (1)-(3) holds as wanted.
\end{proof}

Recall that $G=AB$ is said to be a factorization of $G$ if $A$ and $B$ are proper subgroups of $G$.\par

\begin{corollary}\label{E} If a group $G\in U^*_{\pi,\,\pi'}\cap E_{\pi,\,\pi'}$, then for every composition factor $D$ of $G$, one of the following holds:\par\smallskip
$(1)$ either $\pi(D)\subseteq \pi$ or $\pi(D)\subseteq \pi'$.\par\smallskip
$(2)$ $D$ is isomorphic to $PSL(2,p)$, where $p$ is a Mersenne prime with $p>3$ and either $ \pi\cap \pi(D)=\{2\}$ or $ \pi'\cap \pi(D)=\{2\}$.\par\smallskip
$(3)$ $D$ is isomorphic to $PSL(2,7)$, and either $ \pi\cap \pi(D)=\{7\}$ or $ \pi'\cap \pi(D)=\{7\}$.\par\smallskip
\noindent Also, a group $G\in \widehat{U}^*_{\pi,\,\pi'}\cap D_{\pi,\,\pi'}$ if and only if for every composition factor $D$ of $G$, one of the above statements $(1)$ and $(2)$ holds.
\end{corollary}

\begin{proof}Firstly, suppose that $G\in U^*_{\pi,\,\pi'}\cap E_{\pi,\,\pi'}$. Then by Lemma \ref{hal}(1), $D\in U^*_{\pi,\,\pi'}\cap E_{\pi,\,\pi'}$ for every composition factor $D$ of $G$. We may let $\pi(D)\nsubseteq \pi$ and $\pi(D)\nsubseteq \pi'$. If $D\in E_{2'}$, then by \cite[Theorem 4.2]{Ara}, $D\cong PSL(2,p)$, where $p$ is a Mersenne prime with $p > 3$. Note also that $PSU(3,q)$ admits no factorization except $q=3$ or $5$ by \cite[Theorem 2]{Bla}. This implies that the statement (2) or (3) in Theorem \ref{E3} occurs. Suppose that $D\cong PSL(3,q)$, where $q=2^{f}>2$ with $f\equiv \pm 1\,(mod\,6)$, and either $\pi\cap \pi(D)\subseteq \{2,3\}$ or $\pi'\cap \pi(D)\subseteq \{2,3\}$. In view of the above argument, $D\notin E_{2'}$. If $D\in E_{3'}$, then $D\cong PSL(2,8)$ by \cite[Corollary 5.6]{Ara1}. This contradiction forces that $D\in E_{\{2,\,3\}}$ and $D\in E_{\{2,\,3\}'}$. It follows from \cite[Theorem 1]{Bla} that $\pi(q^3(q+1)(q-1)^2)=\{2,3\}$, and so $q-1$ and $q+1$ are both powers of $3$. Thus $q=2$, a contradiction. Now assume that $D\cong PSL(2,p)$, where $p$ is a Mersenne prime with $p > 3$. Let $H$ be a Hall subgroup of $D$ having order divisible by $2$. If $H$ is not a Sylow $2$-subgroup of $D$, then according to L. E. Dickson's list of subgroups of $PSL(2,p)$, $D\cong PSL(2,7)$ and $H$ is a Hall $\{2,3\}$-subgroup of $D$ by arguing similarly as in the proof of Theorem A. Therefore, one of the statements $(1)$-$(3)$ holds as wanted.

Note that by Lemma \ref{hal}(3), $G\in \widehat{U}^*_{\pi,\,\pi'}\cap D_{\pi,\,\pi'}$ if and only if $D\in \widehat{U}^*_{\pi,\,\pi'}\cap D_{\pi,\,\pi'}$ for every composition factor $D$ of $G$. If $D\in \widehat{U}^*_{\pi,\,\pi'}\cap D_{\pi,\,\pi'}$, then by Theorem \ref{Gi}, the statement (3) does not hold. Also, if the statement (1) or (2) holds, then $D\in \widehat{U}^*_{\pi,\,\pi'}\cap D_{\pi,\,\pi'}$ by Theorems \ref{Gi} and \ref{D2}. The proof is now complete.
\end{proof}

Theorem \ref{hh}, Corollary \ref{E}  and Burnside's $p^aq^b$-theorem guarantee the correctness of the following corollary.\par

\begin{corollary}\label{cor3} $(1)$ If $r\notin \{2,7\}$, then a group $G$ is $r$-soluble if and only if $G\in U^*_{r,\,r'}\cap E_{r,\,r'}$.\par\smallskip
$(2)$ If $r\neq 2$, then a group $G$ is $r$-soluble if and only if $G\in \widehat{U}^*_{r,\,r'}\cap D_{r,\,r'}$.
\end{corollary}

By Theorem \ref{tu} and the Feit-Thompson theorem, the proof of the next lemma is routine, and hence is omitted.\par

\begin{lemma}\label{Fei} $(1)$ $V^*_{\pi,\,\pi'}\subseteq U^*_{\pi,\,\pi'}$. If $2\notin \pi$, then $V^*_{\pi,\,\pi'}\subseteq U^*_{\sigma,\,\sigma'}$ for any set of primes $\sigma$ with $\sigma\subseteq \pi$.\par\smallskip
$(2)$ If $2\notin \pi$, $G\in V^*_{\pi,\,\pi'}$ and $\pi\nsubseteq \pi(G)$, then $G\in {E}_{\pi,\,\pi'}$.\par\end{lemma}

As S. A. Chunikhin introduced, a group $G$ is called $\pi$-selected if the order of every composition factor (or equivalently, every chief factor) of $G$ is divisible by at most one prime in $\pi$. Bearing in mind that $\{\pi\mbox{-selected groups}\}\subseteq D_\sigma$ for any set of primes $\sigma$ with $\sigma\subseteq \pi$ by \cite[Corollary D5.2]{Hal}.\par

\begin{theorem}\label{E4} \textit{$V^*_{\pi,\,\pi'}\subseteq U^*_{\pi,\,\pi'}$. Also, if a group $G\in V^*_{\pi,\,\pi'}$, then $G$ is $\pi'$-selected when $2\in \pi$ and $G$ is $\pi$-selected when $2\notin \pi$. More precisely, for every composition factor $D$ of $G$, one of the following holds:}\par\smallskip
\textit{$(1)$ either $\pi(D)\subseteq \pi$ or $\pi(D)\subseteq \pi'$.}\par\smallskip
\textit{$(2)$ $D$ is isomorphic to $PSL(2,7)$, and either $\pi\cap \pi(D)=\{7\}$ or $\pi'\cap \pi(D)= \{7\}$.}\par\smallskip
\textit{$(3)$ $D$ is isomorphic to $PSL(2,p)$, where $p$ is a Mersenne prime with $p>3$, and either $\pi=\{r\}$ or $\pi'=\{r\}$, where $r\in \pi(\frac{p(p-1)}{2})$.}\par\smallskip
\textit{$(4)$ $D$ is isomorphic to $PSL(3,q)$, where $q=2^{f}>2$ with $f\equiv \pm 1\,(mod\,6)$, and either $\pi=\{3\}$ or $\pi'=\{3\}$.}\par\smallskip
\textit{$(5)$ $D$ is isomorphic to $PSU(3,q)$, where $q\equiv 4$ or $7$ $(mod\, 9)$, and either $\pi=\{3\}$ or $\pi'=\{3\}$.}\par
\end{theorem}

\begin{proof} By Lemma \ref{Fei}(1), $V^*_{\pi,\,\pi'}\subseteq U^*_{\pi,\,\pi'}$. Next, without loss of generality, we may assume that $2\notin \pi$. Let $D$ be any composition factor of $G$ with $\pi(D)\nsubseteq\pi$ and $\pi(D)\nsubseteq\pi'$. Then $D\in V^*_{\pi,\,\pi'}$ by Lemma \ref{hal}(1). It follows from Lemma \ref{fei} that $|\pi\cap \pi(D)|=1$. If $\pi\subseteq \pi(D)$, then $|\pi|=1$. Since $D\in U^*_{\pi,\,\pi'}$, one of the statements (3)-(5) holds by Theorem \ref{E3}, Lemma \ref{he} and Burnside's $p^aq^b$-theorem. Now suppose that $\pi\nsubseteq \pi(D)$. Then by Lemma \ref{Fei}(2), $D\in E_{\pi,\,\pi'}$. Hence by Corollary \ref{E} and Burnside's $p^aq^b$-theorem, $D\cong PSL(2,7)$ and $\pi\cap \pi(D)=\{7\}$, and so the statement (2) follows. Considering together, we have that one of the statements (1)-(5) holds. Also, it is easy to see that $G$ is $\pi$-selected. The proof is thus completed.\end{proof}

Theorems \ref{h} and \ref{E4} and the Feit-Thompson theorem lead to the next corollary.

\begin{corollary}\label{cor20} If $2\notin \pi$, $7\notin \pi$  and $|\pi|\geq 2$, then $G$ is $\pi$-soluble if and only if $G\in V^*_{\pi,\,\pi'}$.
\end{corollary}





\begin{theorem}\label{D3} \textit{$\widehat{V}^*_{\pi,\,\pi'}\subseteq\widehat{U}^*_{\pi,\,\pi'}$. Also, a group $G\in \widehat{V}^*_{\pi,\,\pi'}$ if and only if for every composition factor $D$ of $G$, one of the following holds:}\par\smallskip
\textit{$(1)$ either $\pi(D)\subseteq \pi$ or $\pi(D)\subseteq \pi'$.}\par\smallskip
\textit{$(2)$ $D$ is isomorphic to $PSL(2,p)$, where $p$ is a Mersenne prime with $p>3$, and either $\pi=\{r\}$ or $\pi'=\{r\}$, where $r\in \pi(\frac{p(p-1)}{2})$.}\par\smallskip
\textit{$(3)$ $D$ is isomorphic to $PSU(3,p)$, where $p=7$ or $p=2^f-1$ is a Mersenne prime with $f\equiv 5$ $(mod\, 6)$, and either $\pi=\{3\}$ or $\pi'=\{3\}$.}\par
\end{theorem}

\begin{proof} If $G\in \widehat{V}^*_{\pi,\,\pi'}$, then for every composition factor $D$ of $G$, $D\in \widehat{V}^*_{\pi,\,\pi'}$ by Lemma \ref{hal}(3). Hence by Lemma \ref{fei}, either $|\pi\cap \pi(D)|\leq 1$ or $|\pi'\cap \pi(D)|\leq1$. Then clearly, $D\in \widehat{U}^*_{\pi,\,\pi'}$. Thus $G\in \widehat{U}^*_{\pi,\,\pi'}$ by Lemma \ref{hal}(3), which indicates that $\widehat{V}^*_{\pi,\,\pi'}\subseteq\widehat{U}^*_{\pi,\,\pi'}$. If $D$ is isomorphic to $PSL(2,7)$, and either $\pi\cap \pi(D)=\{7\}$ or $\pi'\cap \pi(D)= \{7\}$, then since $D\notin D_{\{2,\,3\}}$, either $\pi=\{7\}$ or $\pi'= \{7\}$. Thereupon, we can deduce from Theorems \ref{D2} and \ref{E4} and Burnside's $p^aq^b$-theorem that one of the statements (1)-(3) holds. Conversely, if one of the statements (1)-(3) holds, then $G\in \widehat{V}^*_{\pi,\,\pi'}$ by \cite[Theorem 2]{Rev1}. This proves the theorem.
\end{proof}

According to the above theorem, we can give a sufficient and necessary condition for $\pi$-separability.

\begin{corollary}\label{cor21} If $|\pi|\geq 2$ and $|\pi'|\geq 2$, then $G$ is $\pi$-separable if and only if $G\in \widehat{V}^*_{\pi,\,\pi'}$.
\end{corollary}

The following three corollaries are obvious by Theorems \ref{E4} and \ref{D3} and Corollary \ref{E}.\par

\begin{corollary} If a group $G\in V^*_{\pi,\,\pi'}\cap E_{\pi,\,\pi'}$, then for every composition factor $D$ of $G$, one of the following holds:\par\smallskip
$(1)$ either $\pi(D)\subseteq \pi$ or $\pi(D)\subseteq \pi'$.\par\smallskip
$(2)$ $D$ is isomorphic to $PSL(2,7)$, and either $\pi\cap \pi(D)=\{7\}$ or $\pi'\cap \pi(D)=\{7\}$.\par\end{corollary}

\begin{corollary} If $r\neq 7$, then a group $G$ is $r$-soluble if and only if $G\in V^*_{r,\,r'}\cap E_{r,\,r'}$.
\end{corollary}

\begin{corollary} A group $G$ is $\pi$-separable if and only if $G\in \widehat{V}^*_{\pi,\,\pi'}\cap D_{\pi,\,\pi'}$.\par\end{corollary}

At the end of this section, we draw the readers' attention to the following general problem.

\begin{problem}\label{o} Let $\sigma=\{\pi_i\,|\,i\in I, \pi_i\neq \varnothing\}$ be a partition of $\pi$, that is, $\pi=\bigcup_{i\in I}\pi_i$ and $\pi_i\cap \pi_j=\varnothing$ for all $i\neq j$.\par\smallskip
$(1)$ What about the structure of a group $G$ in which for all $i$, $G\in E_{\{r,\,s\}}$ $($or stronger, $G\in D_{\{r,\,s\}}$$)$ for every prime $r\in \pi_i$ and every prime $s\in \pi\backslash \pi_i$$?$\par\smallskip
$(2)$ What about the structure of a group $G$ in which for all $i$, $G\in E_{\{\pi_i,\,s\}}$ $($or stronger, $G\in D_{\{\pi_i,\,s\}}$$)$ for every prime $s\in \pi\backslash \pi_i$$?$
\end{problem}

Theorems A, B and C answer Problem \ref{o} when $\pi=\mathbb{P}$ and $|\sigma|=2$, and theorems in this section answer Problem \ref{o} when $\pi=\mathbb{P}\backslash\{2\}$ and $|\sigma|=2$. Further, by applying these theorems, it is easy to give a solution to Problem \ref{o} when $\pi=\mathbb{P}$ or $\pi=\mathbb{P}\backslash\{2\}$. However, Problem \ref{o} is still open in general. Here we quote another problem proposed by F. Gross in \cite{Gro4} (which was also proposed by A. A. Buturlakin in \cite{Vdo3}).

\begin{problem}\label{oo} Is it true that $G\in E_{\pi}$ if $G\in E_{\{r,\,s\}}$ for every prime $r,s\in \pi$$?$ 
\end{problem}

Note that Problem \ref{oo} can be viewed as a special case of Problem \ref{o} with $|\pi_i|=1$ for all $i$. Though no counterexample was found, a positive solution to Problem \ref{oo} has not been established as far as we are aware.

\end{document}